\documentclass[11pt]{article}
\usepackage[pagebackref=false,colorlinks,linkcolor=green,citecolor=magenta]{hyperref}
\usepackage{amssymb,amsmath}
\usepackage{amsmath}
\usepackage{dsfont}
\usepackage{amsfonts}
\newtheorem{precor}{{\bf Corollary}}

\newenvironment{cor}{\begin{precor}{\hspace{-0.5
               em}{\bf.\ }}}{\end{precor}}
\newtheorem{precon}{{\bf Conjecture}}

\newenvironment{con}{\begin{precon}{\hspace{-0.5
               em}{\bf.\ }}}{\end{precon}}
\newtheorem{prealphcon}{{\bf Conjecture}}

\newtheorem{predefin}{{\bf Definition}}

\newtheorem{preexm}{{\bf Example}}

\newtheorem{preappl}{{\bf Application}}

\newtheorem{prelem}{{\bf Lemma}}

\newenvironment{lem}{\begin{prelem}{\hspace{-0.5
               em}{\bf.\ }}}{\end{prelem}}
\newtheorem{preproof}{{\bf Proof.\ }}

\newenvironment{proof}[1]{\begin{preproof}{\rm
               #1}\hfill{$\blacksquare$}}{\end{preproof}}
\newtheorem{prethm}{{\bf Theorem}}

\newenvironment{thm}{\begin{prethm}{\hspace{-0.5
               em}{\bf.\ }}}{\end{prethm}}
\newtheorem{prealphthm}{{\bf Theorem}}

\newenvironment{alphthm}{\begin{prealphthm}{\hspace{-0.5
               em}{\bf.\ }}}{\end{prealphthm}}

\newtheorem{prealphlem}{{\bf Lemma}}

\newenvironment{alphlem}{\begin{prealphlem}{\hspace{-0.5
               em}{\bf.\ }}}{\end{prealphlem}}

\newtheorem{prepro}{{\bf Proposition}}

\newtheorem{preprb}{{\bf Problem}}

\newtheorem{prerem}{{\bf Remark}}

\newtheorem{preapp}{{\bf Application}}

\newtheorem{prequ}{{\bf Question}}

\newtheorem{preclaim}{{\bf Claim}}

\def\conct[#1,#2]{\mbox {${#1} \leftrightarrow {#2}$}}
\def\dconct[#1,#2]{\mbox {${#1} \rightarrow {#2}$}}
\def\deg[#1,#2]{\mbox {$d_{_{#1}}(#2)$}}
\def\mindeg[#1]{\mbox {$\delta_{_{#1}}$}}
\def\maxdeg[#1]{\mbox {$\Delta_{_{#1}}$}}
\def\outdeg[#1,#2]{\mbox {$d_{_{#1}}^{^+}(#2)$}}
\def\minoutdeg[#1]{\mbox {$\delta_{_{#1}}^{^+}$}}
\def\maxoutdeg[#1]{\mbox {$\Delta_{_{#1}}^{^+}$}}
\def\indeg[#1,#2]{\mbox {$d_{_{#1}}^{^-}(#2)$}}
\def\minindeg[#1]{\mbox {$\delta_{_{#1}}^{^-}$}}
\def\maxindeg[#1]{\mbox {$\Delta_{_{#1}}^{^-}$}}

\def\dre[#1,#2,#3]{\mbox {${\cal E}^{^{#3}}(#1,#2)$}}
\def\var[#1,#2]{\mbox {${\rm Var}_{_{#1}}(#2)$}}
\def\ls[#1]{\mbox {$\xi^{^{#1}}$}}
\def\hom[#1,#2]{\mbox {${\rm Hom}({#1},{#2})$}}
\def\onvhom[#1,#2]{\mbox {${\rm Hom^{v}}(#1,#2)$}}
\def\onehom[#1,#2]{\mbox {${\rm Hom^{e}}(#1,#2)$}}
\def\core[#1]{\mbox {$#1^{^{\bullet}}$}}
\def\cay[#1,#2]{\mbox {${\rm Cay}({#1},{#2})$}}
\def\sch[#1,#2,#3]{\mbox {${\rm Sch}({#1},{#2},{#3})$}}
\def\cays[#1,#2]{\mbox {${\rm Cay_{s}}({#1},{#2})$}}
\def\dirc[#1]{\mbox {$\stackrel{\rightarrow}{C}_{_{#1}}$}}
\def\cycl[#1]{\mbox {${\bf Z}_{_{#1}}$}}

\begin{document}
\footnotetext[1]{The research of Hossein Hajiabolhassan is supported by ERC advanced grant GRACOL.}
\begin{center}
{\Large \bf  On The Chromatic Number of Matching  Graphs}\\
\vspace{0.3 cm}
{\bf Meysam Alishahi$^\dag$ and Hossein Hajiabolhassan$^\ast$\\
{\it $^\dag$ Department of Mathematics}\\
{\it University of Shahrood, Iran}\\
{\tt meysam\_alishahi@shahroodut.ac.ir}\\
{\it $^\ast$ Department of Mathematical Sciences}\\
{\it Shahid Beheshti University, G.C.}\\
{\it P.O. Box {\rm 19839-63113}, Tehran, Iran}\\
{\tt hhaji@sbu.ac.ir}\\
}
\end{center}
\begin{abstract}
\noindent
In an earlier paper, the present authors (2013)~\cite{Alishahi2015}
introduced  the altermatic number of graphs and used
Tucker's Lemma, an equivalent combinatorial version of the Borsuk-Ulam Theorem, to show that
the altermatic number is a lower bound for the chromatic number.
A matching graph has the set of all matchings of a specified size of a graph as vertex set and
two vertices are adjacent if the corresponding matchings are edge-disjoint. It is known that the Kneser graphs, the Schrijver graphs, and
the permutation graphs can be represented by matching graphs. In this paper,
as a generalization of the well-known result of Schrijver about the chromatic number of Schrijver graphs,
we determine the chromatic number of a large family of matching graphs
by specifying their altermatic number. In particular, we determine the chromatic number
of these matching graphs in terms of the generalized Tur\'an number of matchings.\\

\end{abstract}
\section{Introduction}
For a hypergraph ${\cal H}$, the vertex set of {\it the general Kneser graph} ${\rm KG}({\cal H})$ is the set of all hyperedges of
${\cal H}$ and two vertices are adjacent if the corresponding hyperedges are disjoint.
It is known that any graph can be represented
by various  general Kneser graphs. In this paper, we show that one
can determine the chromatic number of some graphs, by choosing
some appropriate representations for them. In view of Tucker's Lemma, an equivalent combinatorial version of the Borsuk-Ulam Theorem,
the present authors~\cite{Alishahi2015} introduced the altermatic number
of Kneser hypergraphs. Moreover, they showed that it provides a tight lower bound for
the chromatic number of Kneser hypergraphs.
Next, they determined the chromatic number of  some families of graphs, e.g.,
Kneser multigraphs, see~\cite{2014arXiv1401.0138A}.

A {\it matching graph} can be
represented by the general Kneser graph ${\rm KG}({\cal H})$,
where the hyperedges of ${\cal H}$ are corresponding to
all matchings of a specified size of a graph.
It is worth noting that a Schrijver graph is a matching graph ${\rm KG}({\cal H})$, where
the hyperedge set of ${\cal H}$ is corresponding to all matchings of a specified size of a cycle.
Hence, by determining the chromatic number of matching graphs,
we generalize the well-known result of Schrijver about the chromatic number of Schrijver graphs.

This paper is organized as follows. In Section~\ref{secnotation}, we set up
notations and terminologies. In particular, we will be concerned
with the definition of altermatic number and strong altermatic
number and we mention some results about them. Also, we introduce
the alternating Tur\'an number of graphs which is a
generalization of the generalized Tur\'an number. Next, we
introduce a lower bound for the chromatic number in terms of the
alternating Tur\'an number. In fact, if one can show that the
alternating Tur\'an number is the same as the generalized Tur\'an
number for a family of graphs, then one can determine the
chromatic number of some family of graphs.

In Section~\ref{secmatching}, we study
the chromatic number of matching graphs. In particular, as a generalization of
the well-known result of Schrijver, we specify the chromatic number of a large family of matching graphs
in terms of the generalized Tur\'an number of matchings. Also, as a
consequence, we determine the chromatic number of large permutation graphs.
\section{Notations and Terminologies}\label{secnotation}
In this section, we setup some notations and terminologies.
Hereafter, the symbol $[n]$ stands for the set $\{1,2,\ldots, n\}$.
A {\it hypergraph} ${\cal H}$ is an ordered pair $(V({\cal H}),E({\cal H}))$,
where $V({\cal H})$ is a set of elements called {\it vertices} and  $E({\cal H})$ is a set of nonempty subsets of $V({\cal H})$
called {\it hyperedges}. Unless otherwise stated, we consider simple hypergraphs, i.e.,
$E({\cal H})$ is a family of distinct nonempty subsets of $V({\cal H})$.
A {\it  vertex cover} $T$ of ${\cal H}$ is a subset
of its vertex set which meets any hyperedge of ${\cal H}$. Also, a {\it $k$-coloring} of a hypergraph ${\cal H}$ is a mapping $h:V({\cal H})\longrightarrow [k]$
such that for any hyperedge $e$, we have $|\{h(v):\ v\in e\}|\geq 2$, i.e., no hyperedge is monochromatic. The minimum $k$ such that
there exists a coloring $h:V({\cal H})\longrightarrow [k]$ is called the
{\it chromatic number} of ${\cal H}$ and is denoted by $\chi({\cal H})$.
Note that if the hypergraph ${\cal H}$ has some hyperedge with cardinality $1$, then there is no
$k$-coloring for any $k$. Therefore, we define the chromatic number of such a hypergraph to be infinite.
A hypergraph ${\cal H}$ is {\it $k$-uniform}, if all hyperedges of ${\cal H}$ have the same size $k$.
By a graph $G$, we mean a 2-uniform hypergraph.
Also, let $o(G)$ denote the number of odd components of the graph $G$.
For brevity, we use $G\cong H$ to mention that there is an isomorphism between
two graphs $G$ and $H$. Also, if $G\cong H$, then we say $G$ and $H$ are {\it isomorphic}.
A {\it homomorphism} from a graph $G$ to a graph $H$
is a mapping $f:V(G)\longrightarrow V(H)$ which preserves the adjacency, i.e.,
if $xy\in E(G)$, then $f(x)f(y)\in E(H)$.
For brevity, we use $G\longrightarrow H$ to denote that there is a homomorphism from
$G$ to $H$. If we have both $G\longrightarrow H$ and $H\longrightarrow G$,
then we say $G$ and $H$ are {\it homomorphically equivalent} and show this by $G\longleftrightarrow H$.
Note that $\chi(G)$ is the minimum integer $k$ for which there is a homomorphism from $G$ to
the complete graph $K_k$. For a subgraph $H$ of $G$, the subgraph $G\setminus H$ is obtained from $G$ by removing the edge set of $H$.
Also, $G-H$ is obtained from $G$ by removing the vertices of $H$ with their incident edges.

\subsection{Altermatic Number}
The sequence $x_1,x_2,\ldots,x_m\in \{-1,+1\}$ is said to be an {\it alternating sequence},
if any two consecutive terms are different.
For any $X=(x_1,x_2,\ldots,x_n)\in\{-1,0,+1\}^n\setminus\{(0,0,\ldots,0)\}$, the {\it alternation number} of
$X$,
$alt(X)$, is the length of a longest alternating subsequence of nonzero terms of
$(x_1,x_1,\ldots,x_n)$. Note that we consider just nonzero entries
to determine the alternation number of $X$. However, for the simplicity of notations, we define $alt((0,0,\ldots,0))=0$.
Let $V=\{v_1,v_2,\ldots,v_n\}$ be a set of size $n$ and $L_V$ be the set of all linear orderings of $V$, i.e.,
$L_V=\left\{v_{i_1}<v_{i_2}<\cdots<v_{i_n}\ :\ (i_1,i_2,\ldots,i_n)\in S_n\right\}$.
For any linear ordering $\sigma: v_{i_1}<v_{i_2}<\cdots<v_{i_n}\in L_V$ and $1\leq j\leq n$, define $\sigma(j)=v_{i_j}$.
We sometimes represent a linear ordering of $V$ by a permutation, i.e., $\sigma=(v_{i_1},v_{i_2},\ldots,v_{i_n})$, and
we use interchangeably these two kinds of representations of any linear ordering. Moreover, for any $X=(x_1,\ldots,x_n)\in\{-1,0,+1\}^{n}$,
define $X^+_\sigma=\{\sigma(j):\ x_j=+1\}=\{v_{i_j}:\ x_j=+1\}$ and
$X^-_\sigma=\{\sigma(k):\ x_k=-1\}=\{v_{i_k}:\ x_k=-1\}$.

For any hypergraph ${\cal H}=(V,E)$ and $\sigma\in L_V$, where $|V|=n$. Define $alt_\sigma({\cal H})$ (resp. $salt_\sigma({\cal H})$)
to be the largest integer $k$ such that there exists an
$X\in\{-1,0,+1\}^n$ with  $alt(X)=k$
and that none (resp. at most one) of $X^+_\sigma$ and $X^-_\sigma$ contains any (resp. some) hyperedge of ${\cal H}$. Note that if any singleton is a hyperedge of ${\cal H}$, then $alt({\cal H})=0$. Also, $alt_\sigma({\cal H})\leq salt_\sigma({\cal H})$ and equality can hold.
Now, set $alt({\cal H})=\min\{alt_\sigma({\cal H}):\ \sigma\in L_V\}$
and $salt({\cal H})=\min\{salt_\sigma({\cal H}):\ \sigma\in L_V\}$. Define the {\it altermatic number}
and the {\it strong altermatic number} of a graph $G$, respectively, as follows
$$\zeta(G)=\displaystyle\max_{\cal H}\left\{|V({\cal H})|-alt({\cal H}):\ {\rm KG}({\cal H})\longleftrightarrow G\right\}$$
and
$$\zeta_s(G)=\displaystyle\max_{\cal H}\left\{|V({\cal H})|+1-salt({\cal H}):\ {\rm KG}({\cal H})\longleftrightarrow G\right\}.$$

It was proved in \cite{Alishahi2015,2014arXiv1401.0138A} that both altermatic number
and  strong altermatic number provide tight lower bounds for the chromatic number of graphs.
\begin{alphthm}{\rm \cite{Alishahi2015}}\label{lowermain}
For any graph $G$, we have
$$\chi(G)\geq \max\left\{\zeta(G),\zeta_s(G)\right\}.$$
\end{alphthm}
\subsection{Alternating Tur{\'a}n Number}
Let $G$ be a graph  and ${\cal F}$ be a family of graphs. 
A subgraph of $G$ is called an {\it ${\cal F}$-subgraph}, if  there is an isomorphism between this subgraph and
a member of  ${\cal F}$. {\it The general Kneser graph} ${\rm KG}(G,{\cal F})$
has all ${\cal F}$-subgraphs of $G$ as vertex set and two vertices are adjacent if
the corresponding ${\cal F}$-subgraphs are edge-disjoint.
A graph $G$ is said to be {\it ${\cal F}$-free}, if it has no subgraph
isomorphic to a member of ${\cal F}$.
For a graph $G$, define ${\rm ex}(G,{\cal F})$, the {\it generalized Tur\'an number of the
family ${\cal F}$ in the graph $G$},
to be the maximum number of edges of an ${\cal F}$-free spanning subgraph of $G$.
A spanning subgraph of $G$ is called {\it ${\cal F}$-extremal} if it has ${\rm ex}(G,{\cal F})$-edges and it is ${\cal F}$-free.
We denote the family of all ${\cal F}$-extremal subgraphs of $G$ with $EX(G, {\cal F})$.
It is usually a hard problem to determine the exact value of ${\rm ex}(G, {\cal F})$.
The concept of Tur\'an number was generalized in \cite{2014arXiv1401.0138A} in order to find the chromatic
number of some families of general Kneser graphs as follows.
Let $G$ be a graph with $E(G)=\{e_1,e_2,\ldots,e_m\}$.
For any ordering $\sigma=(e_{i_1},e_{i_2},\ldots,e_{i_m})$ of edges of $G$,
a $2$-coloring of a subset of edges of $G$ (with $2$ colors red and blue) is
said to be {\it alternating} (respect to the ordering $\sigma$), if
any two consecutive colored edges (with respect to the ordering $\sigma$) receive different colors.
Note that we may assign no color to some edges of $G$. In other words, in view of
the ordering $\sigma$, we assign two colors red and blue alternatively to  a subset of edges of $G$.
We use the notation ${\rm ex}_{alt}(G,{\cal F},\sigma)$
(resp. ${\rm ex}_{salt}(G,{\cal F},\sigma)$) to denote the maximum number of edges of a spanning subgraph $H$ of $G$ such that the edges of $H$ can be colored alternatively
(with respect to the ordering $\sigma$) by $2$-colors so that each (resp. at least one of) color class is ${\cal F}$-free.
Now, we are in a position to define the {\it alternating Tur{\'a}n number} ${\rm ex}_{alt}(G,{\cal F})$
and the {\it  strong alternating Tur{\'a}n number} ${\rm ex}_{salt}(G,{\cal F})$ as follows
$${\rm ex}_{alt}(G,{\cal F})=\min\left\{{\rm ex}_{alt}(G,{\cal F},\sigma);\ \sigma\in L_{E(G)}\right\}$$
and
$${\rm ex}_{salt}(G,{\cal F})=\min\left\{{\rm ex}_{salt}(G,{\cal F},\sigma);\ \sigma\in L_{E(G)}\right\}.$$
For a graph $G$, let $F$ be a member of $EX(G, {\cal F})$ and $\sigma$ be an arbitrary ordering of $E(G)$.
Now, if we color the edges of $F$ alternatively  with two colors with respect to the ordering $\sigma$, one can see that
any color class has~no member of ${\cal F}$; and therefore,
${\rm ex}(G,{\cal F})\leq {\rm ex}_{alt}(G,{\cal F},\sigma)$.
Also, it is clear that if we assign colors to more than $2{\rm ex}(G,{\cal F})$ edges,
then a color class has at least more than ${\rm ex}(G,{\cal F})$
edges. It implies
${\rm ex}_{alt}(G,{\cal F},\sigma) \leq 2{\rm ex}(G,{\cal F})$. Consequently,
$${\rm ex}(G,{\cal F})\leq {\rm ex}_{alt}(G,{\cal F})\leq 2{\rm ex}(G,{\cal F}).$$
Next lemma was proved in \cite{2014arXiv1401.0138A}. For the convenince of the reader, we repeat its proof,
thus making our exposition self-contained.
\begin{alphlem}\label{lowerupper}{\rm \cite{2014arXiv1401.0138A}}
For any graph $G$ and family ${\cal F}$ of graphs,
$$\displaystyle\begin{array}{rllll}
|E(G)|-{\rm ex}_{alt}(G, {\cal F}) &    \leq &  \chi({\rm KG}(G,{\cal F})) & \leq  & |E(G)|-{\rm ex}(G, {\cal F}),\\
& & & &  \\
|E(G)|+1-{\rm ex}_{salt}(G, {\cal F}) &    \leq &  \chi({\rm KG}(G,{\cal F})) & \leq  & |E(G)|-{\rm ex}(G, {\cal F}).
\end{array}
$$
\end{alphlem}
\begin{proof}{
Let $K\in EX(G, {\cal F})$. One can check that any vertex of ${\rm KG}(G,{\cal F})$ contains at least an element of $E(G)\setminus E(K)$. This
implies $\chi({\rm KG}(G,{\cal F})) \leq  |E(G)|-{\rm ex}(G, {\cal F})$. On the other hand,
consider the hypergraph $H$ whose vertex set is $E(G)$ and hyperedge set  consists of all subgraphs of $G$ isomorphic to some member of ${\cal F}$.
Note that ${\rm KG}(H)$ is isomorphic to ${\rm KG}(G,{\cal F})$.
One can check that $alt(H)={\rm ex}_{alt}(G, {\cal F})$ and
$salt(H)={\rm ex}_{salt}(G, {\cal F})$. Now, in view of Theorem~\ref{lowermain},
the assertion holds.
}\end{proof}
The previous lemma enables us to find the chromatic number of some families of graphs.
If we present an appropriate ordering $\sigma$ of the edges of $G$ such that
${\rm ex}_{alt}(G,{\cal F})={\rm ex}(G,{\cal F})$ or
${\rm ex}_{salt}(G,{\cal F})-1={\rm ex}(G,{\cal F})$,
then one can conclude that
$$\chi({\rm KG}(G,{\cal F}))=|E(G)|-{\rm ex}(G,{\cal F}).$$
By this observation, the chromatic number of several
families of graphs was determined in~\cite{2014arXiv1401.0138A}.

Hereafter, for a given $2$-coloring of a subset of edges of $G$,
a spanning subgraph of $G$ whose edge set
consists of all edges such that red (resp. blue) color has been assigned to them,
is termed the {\it red subgraph} $G^R$ (resp. {\it blue subgraph} $G^B$).
Furthermore, by abuse of language, any edge of $G^R$ (resp. $G^B$) is termed a red edge (resp. blue edge).
\section{Matching Graphs}\label{secmatching}

In this section, we investigate the chromatic number of graphs via their altermatic number.
First, in Subsection~\ref{subsecschrijver}, we study the chromatic number of the matching graph ${\rm KG}(G,rK_2)$ when
$G$ is a sparse graph. In contrast, in Subsection~\ref{subsecdense}, we determine the chromatic number of  the matching graph ${\rm KG}(G,rK_2)$ provided that $G$ is a large dense graph.
\subsection{A Generalization of Schrijver's Theorem}\label{subsecschrijver}
Matching graphs can be
considered as a generalization of Kneser, Schrijver, and permutation graphs. In fact, one can check that ${\rm KG}(nK_2,rK_2)$, ${\rm KG}(C_n,rK_2)$, and ${\rm KG}(K_{m,n},rK_2)$,
are isomorphic to  Kneser, Schrijver, and permutation graphs, respectively.
Hence, as a generalization of Lov\'asz's Theorem~\cite{MR514625} and Schrijver's Theorem~\cite{MR512648}, it would be of interest to determine the chromatic number of matching graph ${\rm KG}(G,rK_2)$. It seems that for any graph $G$, we usually
have $\chi({\rm KG}(G,rK_2))=|E(G)|-{\rm ex}(G,rK_2)$, but the assertion is~not true when $G$ is~not a connected graph. For instance,
note that $\chi({\rm KG}(nK_2,rK_2))=|E(nK_2)|-2{\rm ex}(nK_2,rK_2)$. In
this section, we introduce some sufficient conditions such that the equality $\chi({\rm KG}(G,rK_2))=|E(G)|-{\rm ex}(G,rK_2)$ holds.

A famous generalization of Tutte's Theorem by Berge in 1985,
says that the largest number of vertices saturated by
a matching in $G$ is $\displaystyle\min_{S\subseteq V(G)}\{|V(G)|-o(G-S)+|S|\}$, where
$o(G-S)$ is the number of odd components in $G-S$. For a bipartite graph, we define its odd girth to be infinite.
\begin{thm}\label{grk2}
Let $r\geq 2$ be an integer and $G$ be a connected graph with odd girth at least $g$,
vertex set $V(G)=\{v_1,v_2,\ldots,v_n\}$, and degree sequence
 ${\rm deg}_G(v_1)\geq {\rm deg}_G(v_2)\geq \cdots\geq {\rm deg}_G(v_n)$.
 Moreover, suppose that
 $r\leq \max \{ {g\over 2}, {{\rm deg}_G(v_{r-1})+1\over 4} \}$ and $\{v_1,\ldots,v_{r-1}\}$ forms an independent set.
 If ${\rm deg}_G(v_{r-1})$ is an even integer
 or ${\rm deg}_G(v_{r-1})> {\rm deg}_G(v_r)$, then $\chi({\rm KG}(G,rK_2))=|E(G)|-\displaystyle\sum_{i=1}^{r-1}{\rm deg}_G(v_i)$.
 \end{thm}
\begin{proof}{
Consider the subgraph of $G$ containing  of all edges incident to some $v_i$, for $1\leq i\leq r-1$. This subgraph does~not
have any $r$-matching; consequently,
we have $\chi({\rm KG}(G,rK_2))\leq |E(G)|-\displaystyle\sum_{i=1}^{r-1}{\rm deg}_G(v_i)$. Hence,
it is sufficient to show that $G$ satisfies
$\chi({\rm KG}(G,rK_2))\geq |E(G)|-\displaystyle\sum_{i=1}^{r-1}{\rm deg}_G(v_i)$.
Set $s$ to be the number of $v_i$'s such that ${\rm deg}_G(v_i)$ is an odd integer, where $1\leq i\leq r-1$.
If $G$ is an even graph, set $H=G$; otherwise, add a new vertex $w$ and join it to any
odd vertex of $G$ to obtain the graph $H$. Now,
$H$ has an Eulerian tour $e'_1,e'_2,\ldots,e'_{m}$, where if $G$ is~not an even graph,
we start the Eulerian tour with $w$;
otherwise, it starts with $v_{n}$. Consider the
ordering $(e'_1, e'_2, \ldots, e'_{m})$ and remove all new edges incident with $w$
from this ordering to obtain the ordering $\sigma$ for the edge set of $G$. In other words,
if we traverse the edge $e_i$ before than the edge
$e_j$ in the Eulerian tour, then in the ordering $\sigma$ we have $e_i < e_j$.
Now, consider an alternating coloring (with colors blue and red)
of edges of $G$ with respect to the ordering $\sigma$ of length $t$, where if $s\not=0$, then
$t=1+\displaystyle\sum_{i=1}^{r-1}{\rm deg}_G(v_i)$; otherwise, $t=2+\displaystyle\sum_{i=1}^{r-1}{\rm deg}_G(v_i)$. 
Recall that $G^R$ and $G^B$ are the spanning subgraphs of $G$
whose edge sets consist of all red edges and blue edges, respectively. 
We show that if $s=0$,
then each of $G^R$ and $G^B$ has an $r$-matching; and consequently,
${\rm ex}_{salt}(G,rK_2,\sigma)\leq 1+\displaystyle\sum_{i=1}^{r-1}{\rm deg}_G(v_i)$.
Also, we show that, if $s\not =0$, then $G^R$ or $G^B$ has an $r$-matching; and consequently,
${\rm ex}_{alt}(G,rK_2,\sigma)\leq \displaystyle\sum_{i=1}^{r-1}{\rm deg}_G(v_i)$. In view of Lemma \ref{lowerupper}, these, these imply
$\chi({\rm KG}(G,rK_2))= |E(G)|-\displaystyle\sum_{i=1}^{r-1}{\rm deg}_G(v_i)$.
In view of the ordering $\sigma$, one can see that for any vertex $x$ of $G$
(except at most the first
vertex of the Eulerian tour), color red (resp. blue) can be assigned to at most half of
edges adjacent to $x$, i.e., $\lceil  {{\rm deg}_G(x)\over2}\rceil$. For the first vertex,
if $t$ is an even integer, then any color can be assigned to at most half of
edges incident to it. This amount can be increased by at most one when $t$ is an odd integer.
Now, the proof falls into two parts.

\begin{enumerate}
 \item[a{\rm )}] First, we show that if $r\leq {g\over 2}$, then the assertion holds. In fact,
 if $s\not=~0$, then we determine the chromatic number by evaluating the altermatic chromatic number. Otherwise,
we show that the strong altermatic chromatic number is equal to chromatic number. For
$j\in\{R,B\}$, if $G^j$ does~not have any $r$-matching,
then in view of the Tutte-Berge Formula,
there exists an $S^j\subseteq V(G^j)=V(G)$ such that $|V(G^j)|-o(G^j-S^j)+|S^j|\leq 2r-2$.
Suppose that $O^j_1,O^j_2,\ldots,O^j_{t_j}$ are the components of $G^j-S^j$.
One can check that for each $1\leq i\leq t_j$, $|V(O^j_i)|\leq 2r-|S^j|-1\leq g-1$.
Therefore, every component $O^j_i$ does~not have any odd cycle and so
it would be a bipartite graph.
Assume that $O^j_i=O^j_i(X^j_i,Y^j_i)$ such that $|X^j_i|\leq|Y^j_i|$ ($X^j_i$ may be an empty set).
Set $X^j=\displaystyle\cup_{i=1}^{t_j}X^j_i$.
Note that
$$|X^j|=\displaystyle\sum_{i=1}^{t_j}\lfloor{|V(O^j_i)|\over 2}\rfloor\leq {|V(G^j)|-|S^j|\over 2}-{|V(G^j)|+|S^j|-2r+2\over 2}
= r-|S^j|-1.$$
Therefore, for $s=0$ and any $j\in \{R,B\}$, if $G^j$ does~not have any $r$-matching, then
$$\displaystyle\begin{array}{lll}
1+\displaystyle\sum_{i=1}^{r-1}{{\rm deg}_G(v_i)\over 2} &    = & |E(G^j)| \\
& & \\
         & \leq & \displaystyle\sum_{x\in S^j}\lceil{{\rm deg}_G(x)\over 2}\rceil+
                  \displaystyle\sum_{x\in X^j}\lceil{{\rm deg}_G(x)\over 2}\rceil\\
                  & & \\
         & \leq  &\displaystyle\sum_{i=1}^{r-1}{{\rm deg}_G(v_i)\over 2}
\end{array}$$
which is impossible. This means that ${\rm ex}_{salt}(G,rK_2,\sigma)\leq 1+\displaystyle\sum_{i=1}^{r-1}{\rm deg}_G(v_i)$.
Accordingly,  $\chi({\rm KG}(G,rK_2))= |E(G)|-\displaystyle\sum_{i=1}^{r-1}{\rm deg}_G(v_i)$.
Now, assume that $s\not=~0$. Also, suppose that neither $G^R$ nor $G^B$ has any r-matching.
In view of the assumption, ${\rm deg}_G(v_{r-1})$ is an even integer or
${\rm deg}_G(v_{r-1})> {\rm deg}_G(v_r)$. Hence,
$$\displaystyle\begin{array}{lll}
1+\displaystyle\sum_{i=1}^{r-1}{{\rm deg}_G(v_i)} &    = & |E(G^R)|+ |E(G^B)| \\
         & \leq & \displaystyle\sum_j\sum_{x\in S^j} {{\rm deg}_{G^j}(x)}+
                  \displaystyle\sum_j\sum_{x\in X^j}{{\rm deg}_{G^j}(x)}\\
         & \leq  &\displaystyle\sum_{i=1}^{r-1}{{\rm deg}_G(v_i)}
\end{array}$$
which is a contradiction. This means that
${\rm ex}_{alt}(G,rK_2,\sigma)\leq \displaystyle\sum_{i=1}^{r-1}{\rm deg}_G(v_i)$.
Accordingly,  $\chi({\rm KG}(G,rK_2))= |E(G)|-\displaystyle\sum_{i=1}^{r-1}{\rm deg}_G(v_i)$.

\item[b{\rm )}] Now, we show that if
$r\leq {{\rm deg}_G(v_{r-1})+1\over 4}$,  then the assertion holds.
For $j\in\{R,B\}$, if $G^j$ does~not have any $r$-matching, in view of
the Tutte-Berge Formula,
there exists an $S^j\subseteq V(G^j)$ such that $|V(G^j)|-o(G^j-S^j)+|S^j|\leq 2r-2$.
Assume that $O^j_1,O^j_2,\ldots,O^j_{t_j}$ are all components of $G^j-S^j$, where ${t_j}\geq o(G^j-S^j)$.
We consider two different cases $s=0$ and $s>0$.
First assume that $s=0$ and so $t=2+\displaystyle \sum_{i=1}^{r-1}{{\rm deg}_G(v_i)}$.
Now, we show that each of $G^R$ and $G^B$ has a matching of size $r$.
Assume that $G^R$ (resp. $G^B$) does~not have any matching of size $r$.
Therefore,
$$\begin{array}{lll}
|E(G^R)| & \leq & \displaystyle\sum_{x\in S^R}{\rm deg}_{G^R}(x)+\displaystyle\sum_{i=1}^{t_R} {|V(O^R_i)|\choose 2}\\
              & \leq & \displaystyle\sum_{x\in S^R}{\rm deg}_{G^R}(x)+{\displaystyle\sum_{i=1}^{t_R}|V(O^R_i)|-({t_R}-1)\choose 2}\\
              & \leq & \displaystyle\sum_{x\in S^R}{\rm deg}_{G^R}(x)+{2r-2|S^R|-1\choose 2}\\
              & \leq & {1\over 2}\displaystyle\sum_{i=1}^{r-1}{\rm deg}_{G}(v_i)
\end{array}$$
which is impossible. Now, assume that $s\not=0$. We show that $G^R$ or $G^B$ has a matching of size $r$.
On the contrary, suppose that both $G^R$ and $G^B$ does not have any matching of size $r$.
First, suppose that  $|S^R|\not=r-1$ or $|S^B|\not=r-1$.
Note that ${1+\displaystyle\sum_{i=1}^{r-1}{{\rm deg}_G(v_i)}}= t=|E(G^R)|+|E(G^B)|$. On the other hand,
$$\begin{array}{lll}
     t   & \leq & \displaystyle\sum_j\displaystyle\sum_{x\in S^j}  {{\rm deg}_{G^j}(x)}
                        +\displaystyle\sum_j\sum_{i=1}^{t_j} {|V(O^j_i)|\choose 2}\\
        & \leq & \displaystyle\sum_j\displaystyle\sum_{x\in S^j}  {{\rm deg}_{G^j}(x)}+
                           \displaystyle\sum_j{\displaystyle\sum_{i=1}^{t_j}|V(O^j_i)|-({t_j}-1)\choose 2}\\
        & \leq & \displaystyle\sum_j\sum_{i=1}^{|S^j|}  {{\rm deg}_G(v_i)\over 2}+\min\{{r-1\over2},{|S^R|+|S^B|\over 2}\} + \sum_j{2r-2|S^j|-1\choose 2}\\
        & \leq & \displaystyle\sum_j\sum_{i=1}^{|S^j|} {{\rm deg}_G(v_i)\over 2}+\sum_j(r-|S^j|-1){{\rm deg}_G(v_{r-1})\over 2}\\
        & \leq &  \displaystyle\sum_{i=1}^{r-1} {{\rm deg}_G(v_i)}
\end{array}$$
which is impossible. If $|S^R|=|S^B|=r-1$, then each connected component of $G^j-S^j$ is a single vertex.
Hence,
$$ {1+\displaystyle\sum_{i=1}^{r-1}{{\rm deg}_G(v_i)}}\leq |E(G^R)|+|E(G^B)|\leq \displaystyle\sum_j\displaystyle\sum_{x\in S^j}  {{\rm deg}_{G^j}(x)}
\leq \displaystyle\sum_{i=1}^{r-1}  {{\rm deg}_G(v_i)},$$
\end{enumerate}
a contradiction. Consequently, ${\rm ex}_{alt}(G,rK_2,\sigma)\leq \displaystyle\sum_{i=1}^{r-1}{\rm deg}_G(v_i)$; and
accordingly,  $\chi({\rm KG}(G,rK_2))= |E(G)|-\displaystyle\sum_{i=1}^{r-1}{\rm deg}_G(v_i)$.
}\end{proof}
Note that ${\rm KG}(C_n,rK_2)\cong {\rm SG}(n,r)$. Hence, the aforementioned theorem
can be considered as a generalization of Schrijver's Theorem \cite{MR512648}.

\begin{alphthm}{\rm \cite{MR512648}}
For any positive integers $n$ and $r$, where $n\geq 2r$, we have $\chi({\rm SG}(n,r))=n-2r+2$.
\end{alphthm}

\begin{cor}
Let $G$ be  a  connected non-bipartite $k$-regular graph with odd girth at least $g$, where $k$ is an even integer.
For any positive integer $r\leq {g\over 2}$, we have $\chi({\rm KG}(G,rK_2))=|E(G)|-k(r-1)$.
\end{cor}
\begin{proof}{
Let $C$ be a minimal odd cycle in $G$.
Note that $C$ is an induced subgraph of $G$ and $|V(C)|\geq g$. Consequently, it contains an independent set of
size ${\lfloor{g\over 2}\rfloor}$. Therefore, in view of Theorem~\ref{grk2}, the proof is completed.
}\end{proof}
\subsection{Matching-Dense Graphs}\label{subsecdense}
In this subsection, we determine the chromatic number of matching-dense graphs.

Let $G$ be a graph with $V(G)=\{u_1,\ldots,u_n\}$.
The graph $G$ is termed $(r,c)$-{\it locally Eulerian},
if there are edge-disjoint nontrivial Eulerian connected subgraphs $H_1,\ldots,H_n$ of $G$ such that
for any $1\leq i\leq n$,  we have
$u_i\in H_i$ and that for any $u\in V(H_i)$, where $u\not =u_i$, we have
${\rm deg}_{H_i}(u_i) \geq (r-1){\rm deg}_{H_i}(u)+c$.

\begin{lem}\label{matchingdense}
Let $r\geq 2$ and $s$ be nonnegative integers.
Also, let $G$  be a graph with $n$ vertices and $\delta(G) > {r+2\choose 2}+(r-2)s$. If there exists an
$(r+s,c)$-locally Eulerian graph $H$ such that $G$ is a subgraph of $H$, $|V(H)|= |V(G)|+s$, and
$c\geq {r-1\choose 2}+(s+3)(r-1)$, then
$\chi({\rm KG}(G,rK_2))=\zeta({\rm KG}(G,rK_2))=|E(G)|-{\rm ex}(G,rK_2)$.
\end{lem}
\begin{proof}{
In view of Lemma~\ref{lowerupper}, it is sufficient to show that there exists an
ordering $\sigma$ of the edges of $G$ such that
${\rm ex}_{alt}(G,rK_2, \sigma)={\rm ex}(G,rK_2)$.

Let $V(G)=\{u_1,\ldots,u_n\}$ and $V(H)=\{u_1,\ldots, u_{n+s}\}$. In view of definition of  $(r,c)$-locally Eulerian graph,
there are pairwise edge-disjoint nontrivial Eulerian subgraphs $H_1,\ldots,H_{n+s}$ of $H$ such that for any $1\leq i\leq n+s$, we have
$u_i\in H_i$ and that for any $u\in V(H_i)$, where $u\not =u_i$, we have ${\rm deg}_{H_i}(u_i) \geq (r+s-1){\rm deg}_{H_i}(u)+{r-1\choose 2}+(s+3)(r-1)$.

To find the ordering $\sigma$, add a new vertex $x$ and join it to all vertices of $H$ by edges with multiplicity two to obtain the graph  $H'$. Precisely, for any $1\leq i \leq n+s$, join $x$ and $u_i$ with two distinct edges $f_i$ and $f'_i$. Now,
if $H'$ has no odd vertices, then set $\bar{H}=H'$; otherwise,
add a new vertex $z$ to $H'$ and join it to all odd vertices of $H'$ to obtain the graph $\bar{H}$.
 The graph $\bar{H}$ is an even graph; and therefore, it has an Eulerian tour.

 Also, note that the graph $K=\bar{H}-x$ is an even graph; and accordingly, any connected component of
 $K'=K\setminus\left(\displaystyle\bigcup_{i=1}^{n+s} H_i\right)$ is also an Eulerian subgraph. Without loss of generality,
assume that $K_{1},\ldots, K_{l}$ are the connected components of $K'$.
Construct an Eulerian tour for  $\bar{H}$ as follows.
At $i^{\rm th}$ step, where $1\leq i\leq n+s$,
start from the vertex $x$ and traverse the edge $f_i$. Consider an arbitrary Eulerian tour of $H_i$
starting at $u_i$ and traverse it. Next, if there exists a $1\leq j\leq l$ such that $u_i\in K_{j}$ and the edge set of
$K_{j}$ is still untraversed, then consider an Eulerian tour for $K_{j}$ starting at $u_i$ and traverse it. Next,
traverse the edge $f'_i$ and if $i<n+s$, then start $(i+1)^{\rm th}$ step.
Construct an ordering $\sigma$ for the edge set of the graph $G$
such that the ordering of edges in $E(G)$ are
corresponding to their ordering in the aforementioned Eulerian tour, i.e.,
if we traverse the edge $e_i\in E(G)$ before the edge
$e_j\in E(G)$ in the Eulerian tour, then in the ordering $\sigma$ we have $e_i < e_j$.

Now, we  claim that ${\rm ex}_{alt}(G,rK_2, \sigma)={\rm ex}(G,rK_2)$.
Note that for any $r-1$ vertices $\{u_{i_1},\ldots, u_{i_{r-1}}\}\subseteq V(G)$, we have
$${\rm ex}(G,rK_2) \geq \displaystyle\sum_{j=1}^{r-1}{\rm deg}_G(u_{i_j})-{r-1 \choose 2}.$$
Consider an alternating $2$-coloring of the edges of $G$ with
respect to the ordering $\sigma$
of length ${\rm ex}(G,rK_2)+1$, i.e., we assigned  ${\rm ex}(G,rK_2)+1$ blue and red colors alternatively to the edges of $G$
with respect  to the ordering $\sigma$.
For a contradiction, suppose that both red spanning subgraph $G^{R}$ and blue spanning subgraph $G^{B}$
are $rK_2$-free subgraphs.
In view of the Berge-Tutte Formula, there are two sets $T^{R}$ and $T^{B}$ such that
$|V(G^{R})|-o(G^{R}-T^{R})+|T^{R}|\leq 2r-2$ and $|V(G^{B})|-o(G^{B}-T^{B})+|T^{B}|\leq 2r-2$.
In view of these inequalities, one can see that $|T^{R}|\leq r-1$ and $|T^{B}|\leq r-1$. Moreover,
the number of edges of $G^{R}$ not~incident to some vertex of $T^{R}$ ($|E(G^{R}- T^{R})|$)
is at most ${2r-2|T^{R}|-1 \choose 2}$. To see this,
assume that $O^R_1,O^R_2,\ldots,O^R_{t_R}$ are all connected components of $G^R-T^R$,
where ${t_R}\geq o(G^R-T^R)\geq |V(G^{R})|+|T^{R}|-2r+2$. We have
$$\begin{array}{ccc}
|E(G^{R}- T^{R})| \leq  \displaystyle\sum_{i=1}^{t_R} {|V(O^R_i)|\choose 2}& \leq &  \displaystyle{\displaystyle\sum_{i=1}^{t_R}|V(O^R_i)|-({t_R}-1)\choose 2}\\
& & \\
&  \leq & {2r-2|T^R|-1\choose 2}.
\end{array}$$
Similarly, we show $|E(G^{B}- T^{B})| \leq   {2r-2|T^{B}|-1 \choose 2}$.  Also, in view of the ordering $\sigma$, any vertex $u$
of $G$ is incident to at most ${1\over 2}({\rm deg}_G(u)+s+3)$ edges of $G^{R}$ (resp. $G^{B}$).
First, we show that if $|T^R|\leq r-2 $ or $|T^B|\leq r-2$, then the assertion holds. If $|T^R|\leq r-2$, then
$$\begin{array}{lll}
|E(G^R)| & \leq & \displaystyle {2r-2|T^{R}|-1 \choose 2}+\sum_{u\in T^R}{1\over 2}({\rm deg}_G(u)+s+3)\\
&&\\
       & \leq & {2r-2|T^{R}|-1 \choose 2}+{{\rm ex}(G,rK_2)\over 2}-{r-1-|T^{R}|\over 2}\delta(G) +{1\over 2}{r-1 \choose 2}+{(s+3)|T^{R}|\over 2}\\
                &&\\
       &   <  & {{\rm ex}(G,rK_2)\over 2}.
\end{array}$$
Since $|E(G^B)|\leq |E(G^R)|+1$, we have ${\rm ex}_{alt}(G,rK_2, \sigma)=|E(G^R)|+|E(G^B)|\leq {\rm ex}(G,rK_2)$ which is impossible.
Similarly, if $|T^B|\leq r-2$, then the assertion holds.

If $|T^R|=|T^B|=r-1$, in view of  the Berge-Tutte Formula, one can conclude that
all connected components of $G^{R}-T^{R}$ and $G^{B}-T^{B}$ are isolated vertices. This means that
$T^R$ and $T^B$ are vertex covers of $G^R$ and $G^B$, respectively.
Now, we show that $T^{R}=T^{B}$. One the contrary, suppose that $T^{R}\not =T^{B}$.
Without loss of generality, let $|E(G^{R})| \geq |E(G^{B})|$ and choose a vertex $u_i\in T^{R}\setminus T^{B}$.

Set $L^B=|E(H_i^B)|$, i.e., the number of blue edges of $H_i$.
Since $T^B$ is a vertex cover of $G^B$ and all edges of $H_i$ appear consecutively in the ordering $\sigma$  (according to an Eulerian tour of $H_i$), one can check that at least $2L^B$ edges of  $H_i$ are incident to the vertices of $T^B$.
It implies that there is a vertex $z\in V(H_i)\setminus\{u_i\}$ with degree at least ${2L^B\over r-1}$.
However, by the assumption that for any $u\in V(H_i)\setminus \{u_i\}$,
$${\rm deg}_{H_i}(u_i) \geq (r+s-1){\rm deg}_{H_i}(u)+{r-1\choose 2}+(s+3)(r-1),$$
we have
$${\rm deg}_{H_i}(u_i)\geq(r+s-1)\left\lceil{2L^B\over r-1}\right\rceil+{r-1\choose 2}+(s+3)(r-1).$$
Thus, we have
$$L^B\leq{1\over 2}\left({\rm deg}_{H_i}(u_i)-{r-1\choose 2}-(s+3)(r-1)\right).$$

Note that red color can be assigned to at most  one edge between any two consecutive edge of $G$ in the ordering $\sigma$.
Define $l^R$ to be the number of red edges incident to $u_i$ in the subgraph $H_i$.
One can see that $l^R\leq L^B+1$.  If there exists a $1\leq j\leq l$ such that $u_i\in K_{j}$ and $u_r\not \in K_j$ for any $r<i$, then
set $H'_i=H_i\cup K_j$; otherwise, define $H'_i=H_i$.
Now, we show that the number of red edges incident to $u_i$ in the graph $H'_i$ is at most $L^B+1+{1\over 2}{\rm deg}_{H'_{i}}(u_i)-{1\over 2}{\rm deg}_{H_{i}}(u_i)$. If $H'_i=H_i$, then there is nothing to prove. Otherwise, one can check that
red color can be assigned to at most ${1\over 2}{\rm deg}_{H'_{i}}(u_i)-{1\over 2}{\rm deg}_{H_{i}}(u_i)+1$ edges incident to $u_i$ in the subgraph $H'_i\setminus H_i$. Moreover, if $l^R= L^B+1$, then one can see that there are two red edges $e$ and $e'$ in $H_i$
such that $e$ (resp. $e'$) appears before  (resp. after) any edge of $H_i^B$  in the ordering $\sigma$. Consequently, in this case,
red color can be assigned to at most ${1\over 2}{\rm deg}_{H'_{i}}(u_i)-{1\over 2}{\rm deg}_{H_{i}}(u_i)$ edges incident to $u_i$ in the subgraph $H'_i\setminus H_i$. Thus, the number of red edges incident to $u_i$ in the graph $H'_i$ is at most $L^B+1+{1\over 2}{\rm deg}_{H'_{i}}(u_i)-{1\over 2}{\rm deg}_{H_{i}}(u_i)$. Moreover, in view of the ordering of the Eulerian tour of $\bar{H}$, one can conclude that red color can be assigned to at most ${1\over 2}({\rm deg}_G(u_i)+s+1-{\rm deg}_{H'_i}(u_i))$ edges incident to $u_i$ in the subgraph $G\setminus H'_i$. Therefore, the number of  red edges incident to $u_i$ in the graph $G$ is at most
$$\begin{array}{lll}
 L^B+1+{1\over 2}{\rm deg}_{H'_{i}}(u_i)-{1\over 2}{\rm deg}_{H_{i}}(u_i)+ {1\over 2}({\rm deg}_G(u_i)+s+1-{\rm deg}_{H'_i}(u_i)) & \leq & \\
 {1\over 2}({\rm deg}_G(u_i)-(s+3)(r-2)-{r-1\choose 2}) &  &
\end{array}$$
Consequently,
$$\begin{array}{lll}
|E(G^R)|  & \leq &{1\over 2}({\rm deg}_G(u_i)-(s+3)(r-2)-{r-1\choose 2})+ \displaystyle\sum_{u\in T^R\setminus\{u_i\}}{1\over 2}({\rm deg}_G(u)+s+3)\\
                             & & \\
          & \leq & -{1\over 2}{r-1\choose 2}+\displaystyle\sum_{u\in T^R}{1\over 2}{\rm deg}_G(u)\\
\                    & & \\
          & \leq & {1\over 2}{\rm ex}(G,rK_2)
\end{array}$$
and it is a contradiction.
Hence, $T^{R}=T^{B}$. Therefore,
the number of blue and red edges of $G$, i.e., $|E(G^R)|+|E(G^B)|$, is at most the number of edges incident  with vertices in $T^R$
or $T^B$ which is at most ${\rm ex}(G,rK_2)$ and this is impossible.
Accordingly, ${\rm ex}_{alt}(G,rK_2, \sigma)={\rm ex}(G,rK_2)$.
}\end{proof}
Let $G$ be a graph.
A $G$-decomposition of a graph $H$ is a set $\{G_1,\ldots,G_t\}$ of pairwise edge-disjoint subgraphs of $H$ such that for each
$1\leq i\leq t$, the graph $G_i$ is isomorphic to $G$; and moreover, the edge sets of $G_i$'s partition the edge set of $H$.
A $G$-decomposition of $H$ is called a {\it monogamous $G$-decomposition},
if any distinct pair of vertices of $H$ appear in at most one copy of $G$ in the decomposition. Note that if a
graph $H$ has a decomposition into the complete graph $K_t$, then it is a monogamous $K_t$-decomposition.
\begin{alphthm}{\rm \cite{MR1723879}}\label{monogamous}
Let $m$ and $n$ be positive even integers. The complete bipartite graph $K_{m,n}$ has a  monogamous $C_4$-decomposition if and only if $(m,n)=(2,2)$ or $6\leq n\leq m\leq 2n-2$.
\end{alphthm}

\begin{lem}\label{locallyeulerian}
Let $r, t$ and $t'$ be positive integers, where $11\leq t\leq t'\leq 2t-2$.
If $c$ is a nonnegative integer and $t\geq 8r+4c+2$,
then the complete bipartite graph $K_{t,t'}$ is $(r,c)$-locally
Eulerian.
\end{lem}
\begin{proof}{
Assume that $t=2p+q$ and $t'=2p'+q'$, where $0\leq q\leq 1$ and $0\leq q'\leq 1$. Extend the complete bipartite graph
$G=K_{t,t'}$ to the complete bipartite graph $H=K_{T,T'}$, where $T=t+q$ and $T'=t'+q'$.
In view of Theorem~\ref{monogamous},  consider a monogamous $C_4$-decomposition of $H$.
Call any $C_4$ of this decomposition a block, if it is entirely in $G$.
Construct a bipartite graph with the vertex set $(U,V)$ such that $U$ consists of
$\lfloor{t-3\over 8}\rfloor$ copies of each vertex of $K_{t,t'}$
and $V$ consists of all blocks. Join a vertex of $U$ to a vertex of $V$, if the corresponding vertex of
$K_{t,t'}$ is contained in the corresponding block.
One can check that the degree of each vertex in the part $U$ is at least $\lceil{t-3\over 2}\rceil$
and the degree of any vertex in the part $V$ is  $4\lfloor{t-3\over 8}\rfloor$.
In view of Hall's Theorem, one can see that this bipartite graph has a matching which saturates all vertices of $U$. For any
vertex $v\in K_{t,t'}$, consider $\lceil{t-3\over 8}\rceil$ blocks assigned to $v$ through the perfect matching and
set $H_v$ to be a subgraph of $K_{t,t'}$ formed by the union of these blocks.
One can see that ${\rm deg}_{H_v}(v)=2\lfloor{t-3\over 8}\rfloor$, while the degree of
any other vertices of $H_v$ is $2$. In view of $H_v$'s, one can see that  $K_{t,t'}$ is
an $(r,c)$-locally Eulerian graph.
}
\end{proof}
For a family ${\cal F}$ of graphs, we say a graph $G$ has an {\it ${\cal F}$-factor} if there are vertex-disjoint subgraphs $H_1,H_2,\ldots,H_t$
of $G$ such that each $H_i$ is a member of ${\cal F}$ and $\displaystyle\bigcup_{i=1}^t V(H_i)=V(G)$.
Note that if a graph $G$ has an  {\it ${\cal F}$-factor}, where each member of  ${\cal F}$ is an $(r,c)$-locally Eulerian
graph, then $G$ is also $(r,c)$-locally Eulerian. In view of the aforementioned lemma, if a graph $G$ has a $K_{t,t'}$ factor,
then one can conclude that $G$ is $(r,c)$-locally Eulerian. Now, we introduce some sufficient condition for a graph to have a $K_{t,t'}$ factor.

Graph expansion was studied extensively  in the literature.
Let $0<\nu\leq \tau<1$ and assume that $G$ is a graph with $n$ vertices. For $S\subseteq V(G)$, the {\it $\nu$-robust neighborhood of $S$}, $RN_{\nu,G}(S)$, is the set of all vertices  $v\in V(G)$ such that $|N_G(v)\cap S|\geq \nu n$.
A graph $G$ is called {\it robust $(\nu,\tau)$-expander}, if for any $S\subseteq V(G)$ with $\tau n\leq|S|\leq (1-\tau)n$ we have
$|RN_{\nu,G}(S)|\geq |S|+\nu n$. Throughout this section, we write $0<a\ll b\ll c$ to mean that we can choose the constants
$a$, $b$, and $c$ from right to left. More precisely, there are two increasing functions $f$ and $g$
such that, given $c$, whenever we choose some $b\leq g(c)$
and $a\leq f(b)$, for more about robust $(\nu,\tau)$-expander see~\cite{MR3002574}.
A graph $G$ with $n$ vertices has {\it bandwidth} at most $b$
if there exists a bijective assignment $l:V(G)\longrightarrow [n]$ such that for every edge $uv\in E(G)$,
we have $|l(u)-l(v)|\leq b$.
\begin{alphthm}\label{expander}{\rm \cite{MR3002574}}
Let $\nu, \tau$, and $\eta$ be real numbers, where $0<\nu\leq\tau\ll\eta<1$, and $\Delta$ be a positive integer.
There exist constants $\beta >0$ and $n_0$ such that the following holds.
Let  $H$ be a bipartite graph on $n\geq n_0$ vertices with $\Delta(H)\leq \Delta$
and bandwidth at most $\beta n$. If $G$ is a robust $(\nu,\tau)$-expander with $n$ vertices and $\delta(G)\geq \eta n$,
then $G$ contains a copy of $H$.
\end{alphthm}
\begin{thm}\label{expandercondition}
Let  $\nu, \tau$ and $\eta$ be real numbers, where $0<\nu\leq\tau\ll\eta<1$ .
If $n$ is sufficiently large, then for any robust $(\nu,\tau)$-expander graph $G$ with $n$ vertices and $\delta(G)\geq \eta n$,
$\chi({\rm KG}(G,rK_2))=\zeta({\rm KG}(G,rK_2))=|E(G)|-{\rm ex}(G,rK_2)$.
\end{thm}
\begin{proof}{
In view of Lemma~\ref{matchingdense}, it is sufficient to show that the graph $G$ is an
$(r, c)$-locally Eulerian graph, where $c={r-1\choose 2}+3(r-1)$. Set $t= 2r^2+14r-6$ and let $k$ and $t'$ be integers such that
$n =2tk + t'$ and $2t \leq  t' \leq 4t-1$. Now, set $H$ to be a bipartite graph on $n$ vertices
with $1+ {n-t' \over 2t}$ connected components such
that one component is isomorphic to  $K_{\lceil {t'\over 2}\rceil, \lfloor {t'\over 2} \rfloor}$
and any other component is
isomorphic to $K_{t,t}$. In view of Theorem~\ref{expander}, if $n$ is sufficiently large, then $H$ is a spanning subgraph of $G$.
By Lemma~\ref{locallyeulerian}, one can see that $K_{\lceil {t'\over 2}\rceil, \lfloor {t'\over 2} \rfloor}$ and
$K_{t,t}$ are both $(r, c)$-locally
Eulerian graph; consequently, by Lemma \ref{matchingdense}, the theorem follows.
}\end{proof}

The following lemma is an immediate consequence of Lemma 13 from \cite{MR2644240}.
\begin{alphlem}{\rm \cite{MR2644240}}
For positive constants $\tau\ll\eta<1$, there exists an integer $n_0$ such that
if $G$ is a graph with $n\geq n_0$ vertices and the
degree sequence $d_1\leq d_2\leq\cdots\leq d_n$
such that for any $i< {n\over 2}$,  $d_i\geq i+\eta n$ or $d_{n-i-\lfloor\eta n\rfloor}\geq n-i$,
then $\delta(G)\geq \eta n$ and $G$ is a robust $(\tau^2, \tau)$-expander.
\end{alphlem}
In view of the previous lemma and Theorem \ref{expandercondition}, we have the next corollary.
\begin{cor}
For a positive constant $\gamma$, there is an integer $n_0$ such that for any $n\geq n_0$
we have the followings.
If  $G$ is a connected graph with $n$ vertices and the degree sequence
$d_1\leq d_2\leq\cdots\leq d_n$ such that that
for each $i<{n\over 2}$, we have $d_i\geq i+\gamma n$ or $d_{n-i-\lfloor \gamma n\rfloor}\geq n-i$, then $\chi({\rm KG}(G,rK_2))=|E(G)|-{\rm ex}(G,rK_2)$.
\end{cor}
For a graph property ${\cal P}$, we say $G(n,p)$ possesses ${\cal P}$
{\it asymptotically almost sure}, or a.a.s. for brevity,
if the probability that $G\in G(n,p)$ possesses the property ${\cal P}$ tends to $1$
as $n$ tends to infinity.
For constants $0<\nu\ll\tau\ll p<1$, a.a.s. any graph $G$ in $G(n,p)$ is a robust $(\nu,\tau)$-expander
graph with minimum degree at least ${pn\over 2}$ and maximum degree at most $2np$.
This observation and Theorem~\ref{expandercondition} imply that a.a.s. for any graph $G$ in $G(n,p)$
we have $\chi({\rm KG}(G,rK_2))=|E(G)|-{\rm ex}(G,rK_2)$.
Moreover, Huang, Lee, and Sudakov \cite{MR2871764} proved
a more general theorem. Here, we state it in a special case.
\begin{alphthm}{\rm \cite{MR2871764}}
For positive integers $r,\Delta$ and reals $0<p\leq1$ and $\gamma>0$,
there exists a constant $\beta>0$ such that
a.a.s., any spanning subgraph $G'$ of any $G\in G(n,p)$ with minimum degree
$\delta(G')\geq p({1\over 2}+\gamma)n$ contains every
$n$-vertex bipartite graph $H$ which has the maximum degree at most $\Delta$ and bandwidth at most $\beta n$.
\end{alphthm}
In view of the proof of  Theorem \ref{expandercondition} and the previous theorem, we have
the following result.
\begin{cor}
If $0<p\leq 1$ and $\gamma>0$, then a.a.s. for any spanning subgraph
$G'$ of any $G\in G(n,p)$ with minimum degree at least $p({1\over 2}+\gamma)n$
we have $\chi({\rm KG}(G',rK_2))=|E(G')|-{\rm ex}(G',rK_2)$.
\end{cor}
Let $H$ be a graph with $h$ vertices and $\chi(H)=l$.
Set ${\rm cr}(H)$ to be the size of the smallest color class
over all proper $l$-colorings of $H$.
The {\it critical chromatic number}, $\chi_{\rm cr}(H)$, is defined as $(l-1){h\over h-{\rm cr}(H)}$.
One can check that $\chi(H)-1<\chi_{{\rm cr}}(H)\leq \chi(H)$ and equality holds in the upper bound
if and only if in any $l$-coloring of $H$, all color classes have the same size.
Assume that $H$ has $k$ connected components $C_1,C_2,\ldots,C_k$.
Define ${\rm hcf}_c(H)$ to be the highest common factor of integers
$|C_1|,|C_2|,\ldots,|C_k|$. Let $f$ be an $l$-coloring of $H$ such that $x_1\leq x_2\leq\cdots\leq x_l$ are the size of coloring
classes in $f$. Set $D(f)=\{x_{i+1}-x_i|\ 1\leq i\leq l-1\}$
and $D(H)=\displaystyle\bigcup D(f)$ where the union ranges over all $l$-colorings $f$ of $H$.
Now, define ${\rm hcf}_\chi(H)$ to be the highest common factor of the members of $D(H)$.
If $D(H)=\{0\}$, then we define ${\rm hcf}_\chi(H)=\infty$.
We say
$$H\ {\rm is\ in}\ class\ 1\ {\rm if}\left
\{ \begin{array}{ll}
{\rm hcf}_\chi(H)=1                              & {\rm when}\ \chi(H)\neq 2,\\
{\rm hcf}_\chi(H)\leq 2\  {\rm and}\  {\rm hcf}_c(H)=1 & {\rm when}\ \chi(H)= 2,
\end{array}\right.$$
otherwise, $H$ is in {\it Class} $2$.
\begin{alphthm}\label{criticalfactor}{\rm\cite{MR2506388}}
For every graph $H$ on $h$ vertices, there are integers $c$ and $m_0$
such that for all integers $m\geq m_0$, if $G$ is a graph on $n=mh$ vertices,
then the following holds. If
$$\delta(G)\geq\left\{
\begin{array}{ll}
(1-{1\over \chi_{{\rm cr}}(H)})n+c & H\ {\rm is\ in\ Class\ 1},\\
(1-{1\over \chi(H)})n+c      & H\ {\rm is\ in\ Class\ 2},
\end{array}\right.$$
then $G$ has an $H$-factor.
\end{alphthm}

\begin{thm}
For any integer $r\geq 2$, there are constants $\alpha$, $\beta$, and $n_0$ such that
for any graph $G$ with $n$ vertices,
if $\delta(G)\geq ({1\over 2}-\alpha)n+\beta$ and $n\geq n_0$, then $\chi({\rm KG}(G,rK_2))=\zeta({\rm KG}(G,rK_2))=|E(G)|-{\rm ex}(G,rK_2)$.
\end{thm}
\begin{proof}{
Define $t=2r^2+14r-6$.
Set $H$ to be a bipartite graph with two connected components $C_1$ and $C_2$ isomorphic to
$K_{t,t}$ and $K_{t+1,t}$, respectively.
One can check that $H$ is in Class $1$. Hence,
by Theorem \ref{criticalfactor}, there are integers $c_1$ and $m_1$ such that
if $|V(G')|\geq m_1$, $|V(H)|$ divides $|V(G')|$, and $\delta(G') \geq (1-{1\over \chi_{{\rm cr}}(H)})|V(G')|+c_1$, then
the graph $G'$ has an $H$-factor.
Let $T$ be an integer such that $4t+1\leq T<8t+2$ and $4t+1|n-T$. It is known that
if $n$ is sufficiently large and $\delta(G)\geq \eta n$, where $0< \eta <1$, then $G$ contains a copy of
the complete bipartite graph $K_{\lceil{T\over 2}\rceil,\lfloor{T\over 2}\rfloor}$.
Note that ${1\over \chi_{{\rm cr}}(H)}={1\over 2}+{1\over 8t+2}$.
Set $\alpha={1\over 8t+2}$ and $\beta=c_1+8t-1$. If $\delta(G)\geq ({1\over 2}-\alpha)n+\beta$ and
$n$ is sufficiently large, then $G$ contains $K_{\lceil{T\over 2}\rceil,\lfloor{T\over 2}\rfloor}$ and also
the graph $G\setminus K_{\lceil{T\over 2}\rceil,\lfloor{T\over 2}\rfloor}$
has an $H$-factor. Hence, $G$ can be decomposed into the complete bipartite graphs
$K_{t,t}$, $K_{t+1,t}$, and $K_{\lceil{T\over 2}\rceil,\lfloor{T\over 2}\rfloor}$.
In view of Lemma \ref{locallyeulerian}, these graphs
are $(r,c)$-locally Eulerian graphs with
$c={r-1\choose 2}+3(r-1)$.
Therefore, $G$ is an $(r,c)$-locally Eulerian graph; and consequently,
by Lemma \ref{matchingdense}, the assertion holds.
}
\end{proof}
\subsection{Permutation Graphs}
Let $m,n,r$ be positive integers, where $r\leq m,n$.
For an $r$-subset $A\subseteq [m]$ and an injective map $f:A\longrightarrow [n]$, the ordered pair
$(A,f)$ is said to be an $r$-partial permutation \cite{MR2892478}.
Let $S_r(m, n)$ denotes the set of all $r$-partial permutations.
Two partial permutations $(A, f)$ and $(B, g)$ are said to be intersecting, if there exists an
$x\in A\cap B$ such that $f(x) = g(x)$.
Note that $S_n(n, n)$ is the set of all $n$-permutations. The permutation graph $S_r(m, n)$
has all $r$-partial permutations $(A, \sigma)$ as its vertex set and two $r$-partial permutations are adjacent
if and only if they are~not intersecting.
Note that $S_r(m,n)\cong S_r(n,m)$; and therefore, for the simplicity, we assume
that $m\geq n$ for all permutation graphs.
One can see that the permutation graph $S_r(m, n)$ is isomorphic to
${\rm KG}(K_{m,n}, rK_2)$.

Next theorem gives a sufficient condition for a balanced bipartite graph to
have a decomposition into complete bipartite subgraphs.
\begin{alphthm}{\rm \cite{MR2519934}}\label{zhao}
For any integer $q\geq 2$, there exists a positive integer $m_0$ such that for all $m \geq m_0$,
the following holds. If $G=(X,Y)$ is a balanced
bipartite graph on $2n = 2mq$ vertices, i.e., $|X|=|Y|=mq$, with
$$\delta(G) \geq
\left \{ \begin{array}{ll}
{n\over 2}+q-1 & {\rm if}\ m\ {\rm is\ even}\\
{n+3q\over 2} -2  & {\rm if}\ m\ {\rm is\ odd,}
\end{array}\right.$$
then $G$ has  a $K_{q,q}$-factor.
\end{alphthm}

Now, we investigate the chromatic number of general Kneser graph ${\rm KG}(G, rK_2)$ provided
that $G=(X,Y)$ is a balanced ($|X|=|Y|$) dense bipartite graph.
In particular, we determine the chromatic number of
any permutation graph provided that the number of its vertices is large enough.
For more about permutation graphs, see \cite{MR2009400, MR2489272,MR2202076}.
\begin{thm}\label{bipartitematchinggraphs}
For any positive integer $r$, there exist constants $q$ and $m$ such that
for all $n\geq m$, the following holds.
If $G$ is a graph on $2n$ vertices which has a bipartite subgraph $H=(U,V)$
with $|U|=|V|=n$ and $\delta(H) \geq  {n\over 2}+q$, then
$\chi({\rm KG}(G,rK_2))=\zeta({\rm KG}(G,rK_2))=|E(G)|-{\rm ex}(G,rK_2)$.
\end{thm}
\begin{proof}{
In view of Lemma~\ref{matchingdense}, it is sufficient to show that the
graph $H$ is an $(r,c)$-locally Eulerian graph, where $c={r-1\choose 2}+3(r-1)$.
Set $t=2r^2+14r-6$.
By Theorem \ref{zhao}, there are integers $q_1$ and $m_1$
such that if $n\geq m_1$ and $t|n$, then any balanced bipartite graph $H'$ with $2n$ vertices and
$\delta(H')\geq {n\over 2}+q_1$ has a $K_{t,t}$-factor.

Let $t'$ be an integer, where $t\leq t'<2t$ and $t|n-t'$.
It is known that if $n$ is sufficiently large and $\delta(H)\geq \eta n$,
where $0< \eta <1$, then $H$ contains a copy of
the complete bipartite graph $K_{t',t'}$.
Define $q=q_1+2t-1$.
Note that if $n$ is sufficiently large, then $H$ contains a copy of $K_{t','t}$ and also,
in view of Theorem~\ref{zhao}, $H\setminus K_{t',t'}$ has a $K_{t,t}$-factor.
This implies that $H$ can be decomposed into complete bipartite graphs $K_{t',t'}$ and $K_{t,t}$.
In view of Lemma \ref{locallyeulerian}, these graphs
are $(r,c)$-locally Eulerian graphs, where $c={r-1\choose 2}+3(r-1)$.
Therefore, $H$ and $G$ are $(r,c)$-locally Eulerian graphs; and consequently,
by Lemma \ref{matchingdense}, the assertion holds.}\end{proof}
\begin{cor}\label{permutationgraph}
Let $m,n,r$ be positive integers, where $m\geq n \geq r$. If $m$ is large enough, then
$$\chi({\rm KG}(K_{m,n}, rK_2))=\zeta({\rm KG}(K_{m,n},rK_2))=m(n-r+1).$$
\end{cor}
\begin{proof}{
In view of Hall's Theorem, any maximal $rK_2$-free subgraph of $K_{m,n}$ has $(r-1)m$ edges. Hence,
in view of Lemma~\ref{lowerupper}, we have $\chi({\rm KG}(K_{m,n}, rK_2))\leq m(n-r+1)$.
In view of Lemma~\ref{bipartitematchinggraphs}, if $m$ is sufficiently large, then $\chi({\rm KG}(K_{m,m}, rK_2))=m(m-r+1)$.
Now, we show that
for any positive integer $n< m$, if $m$ is sufficiently large, then $\chi({\rm KG}(K_{m,n}, rK_2))=m(n-r+1)$. To see this, on the
contrary, suppose that $f: V({\rm KG}(K_{m,n}, rK_2))\longrightarrow \{1,2,\ldots, \chi({\rm KG}(K_{m,n}, rK_2))\}$ is a proper coloring of
${\rm KG}(K_{m,n}, rK_2)$, where $\chi({\rm KG}(K_{m,n}, rK_2))<m(n-r+1)$.
Add $m-n$ new vertices to the small part of $K_{m,n}$ and join them to all vertices in the other part to construct
$K_{m,m}$, and call the new edges $e_1,\ldots,e_{(m-n)m}$. Extend the coloring $f$ to a proper coloring $g$ for
${\rm KG}(K_{m,m}, rK_2)$ as follows. If a matching $M$ is a subset of $K_{m,n}$, then set $g(M)=f(M)$; otherwise,
assume that $i$ is the smallest positive integer such that $e_i\in M$, in this case set $g(M)=i+\chi({\rm KG}(K_{m,n}, rK_2))$.
This provides a proper coloring for ${\rm KG}(K_{m,m}, rK_2)$ with less than $m(m-r+1)$ colors which is a contradiction.
}\end{proof}
Let $s\geq t$ be positive integers and $G=G(X,Y)$
be a connected $(s,t)$-regular connected bipartite graph. Theorem \ref{grk2} implies that if $s$ is an even integer, then
for any $r\leq |X|$ we have $\chi({\rm KG}(G,rK_2))=s(|X|-r+1)$. This result shows that $\chi({\rm KG}(K_{m,n}, rK_2))=\chi(S_r(m,n))=m(n-r+1)$
provided that $m$ is an even integer and $m\geq n\geq r$. However, if $m$ is a small odd value, then
the chromatic number of the permutation graph
$S_r(m,n)$ is unknown.

\begin{cor}\label{permutationgrapheven}
Let $m,n,r$ be positive integers, where $m\geq n \geq r$. If $m$ is even, then
$$\chi({\rm KG}(K_{m,n}, rK_2))=\zeta({\rm KG}(K_{m,n},rK_2))=m(n-r+1).$$
\end{cor}

The aforementioned results motivate us to consider the following conjecture.
\begin{con}
For any connected graph $G$ and positive integer $r$, we have
$$\chi({\rm KG}(G,rK_2))=|E(G)|-{\rm ex}(G,rK_2).$$
\end{con}

\noindent{\bf Acknowledgement:}
The authors would like to express their deepest gratitude to Professor Carsten~Thomassen for his insightful comments.
They also appreciate the detailed valuable comments of  Dr.~Saeed~Shaebani. This paper was
written while Hossein Hajiabolhassan was visiting School of Mathematics, Institute for Research in Fundamental Sciences~(IPM). He acknowledges the support of IPM. 
Moreover, they would like to thank Skype for sponsoring their endless conversations in two countries.
\def\cprime{$'$} \def\cprime{$'$}

\end{document}